\documentclass{siamltex}

\usepackage{amsmath,amssymb}

\usepackage[dvistyle]{todonotes}
\usepackage{graphicx}

\providecommand{\dd}{\mathrm{d}}
\DeclareMathOperator{\dom}{dom}
\DeclareMathOperator{\lspan}{span}

\newtheorem{remark}[theorem]{Remark}

\begin{document}
\title{Semigroup Splitting And Cubature Approximations For The Stochastic Navier-Stokes Equations\thanks{Financial support from the ETH Foundation is gratefully acknowledged.}}
\author{Philipp D\"orsek\thanks{ETH Z\"urich, D-MATH, R\"amistrasse 101, 8092 Z\"urich, Switzerland (philipp.doersek@math.ethz.ch)}}
\maketitle
\begin{abstract}
  Approximation of the marginal distribution of the solution of the stochastic Navier-Stokes equations on the two-dimensional torus by high order numerical methods is considered.
  The corresponding rates of convergence are obtained for a splitting scheme and the method of cubature on Wiener space applied to a spectral Galerkin discretisation of degree $N$.
  While the estimates exhibit a strong $N$ dependence, convergence is obtained for appropriately chosen time step sizes.
	Results of numerical simulations are provided, and confirm the applicability of the methods.
\end{abstract}
\pagestyle{myheadings}
\thispagestyle{plain}
\markboth{PHILIPP D\"ORSEK}{APPROXIMATION OF STOCHASTIC NAVIER-STOKES EQUATION}

\section{Introduction}
The issue of turbulence in fluid flows is an essentially unsolved problem.
From the perspective of numerical analysis, its main difficulty is that a direct numerical simulation (DNS), resolving all relevant temporal and spatial scales, is unavailable for many practically relevant geometries.
Hence, we can only use results from underresolved simulations, which are often useless due to their severely reduced accuracy.

This has led to the introduction of reduced models that deal with the closure problem, see e.g.~\cite{Pope2000,BerselliIliescuLayton2006}.
These models deal with underresolution by introducing a model for the effects taking place on scales smaller than those that are resolved.

We are concerned with a different approach to turbulence modelling.
In the last years, the introduction of noise into the equations of fluid dynamics has become the focus of research (see e.g.~\cite{BensoussanTemam1973,Kotelenez1995,DaPratoZabczyk1996,MikuleviciusRozovskii2004,AlbeverioFlandoliSinai2008}).
In particular, Hairer and Mattingly proved in \cite{HairerMattingly2006,HairerMattingly2008} that the stochastic Navier-Stokes equations on the two-dimensional torus with finite-dimensional, additive noise have ergodic dynamics, and estimated the rate of convergence to the invariant measure.

There are several ways to find numerical approximations to stochastic differential equations.
Pathwise and strong schemes aim to find path-by-path simulations of the problem, aiming for convergence almost surely or in mean square error; see e.g.~\cite{JentzenKloeden2009,JentzenKloedenNeuenkirch2009} for recent survey articles.
In this work, we consider the approximation of the stochastic Navier-Stokes equations on the two-dimensional torus by weak approximation schemes.
As we are actually searching for estimates of functionals of the invariant measure, this allows us to benefit from the typically higher rate of convergence of weak schemes.
Related work concerning the numerical analysis of weak methods for stochastic partial differential equations can be found in \cite{BayerTeichmann2008,DoersekTeichmann2010}.

In contrast to \cite{HouLuoRozovskiiZhou2006}, we propose a simulation method, based either on a splitting similar to the Ninomiya-Victoir scheme \cite{NinomiyaVictoir2008}, or the method of cubature on Wiener space of Kusuoka \cite{Kusuoka2001} and Lyons and Victoir \cite{LyonsVictoir2004}, extending the applicability of such approaches from bounded vector fields to vector fields that are neither Lipschitz continuous nor linearly bounded anymore (see also \cite{Alfonsi2010} for other such extensions, in that case to vector fields behaving like square roots at zero).
The advantage of such an approach is that it is trivial to parallelise, as every path can be simulated independently.
Furthermore, in the case of splitting schemes, we can furthermore reuse well-tested, robust and fast solvers for the deterministic Navier-Stokes or Euler equations to obtain solvers for the stochastic Navier-Stokes equations with minimal effort.

To derive rates of convergence, we employ the theory derived in \cite{DoersekTeichmann2010}.
While we are unable to prove rates of convergence on the continuous level, a discretisation by a spectral Galerkin scheme allows us to obtain a suboptimal convergence estimate.

This paper is organised as follows.
In Section~\ref{sec:sns}, we recall the definition of the stochastic Navier-Stokes equations in the setting of Hairer and Mattingly and consider them from the perspective of weighted spaces used in \cite{DoersekTeichmann2010}.
Our analysis profits greatly from the fundamental results shown by Hairer and Mattingly in \cite{Mattingly1998,HairerMattingly2006,HairerMattingly2008}.
Section~\ref{sec:spectralgalerkin} is devoted to the derivation of estimates for the error done by a spectral Galerkin approximation.
Section~\ref{sec:rateconv} presents the main results of this paper, estimates for full discretisations of the stochastic Navier-Stokes equations by splitting and cubature schemes.
In Section~\ref{sec:numerics}, we present the results of numerical calculations for a model problem with ergodic dynamics, and in Section~\ref{sec:conclusion}, we sum up our results.

\section{The stochastic Navier-Stokes equations and weighted spaces}
\label{sec:sns}
Consider, as in \cite{HairerMattingly2006,HairerMattingly2008}, the vorticity formulation of the stochastic Navier-Stokes equations on the two-dimensional torus $\mathbb{T}^2$,
\begin{alignat}{2}
	\label{eq:sns-conteq}
	\dd w(t,w_0)
	&=
	\nu\Delta w(t,w_0)\dd t + B(\mathcal{K}w(t,w_0),w(t,w_0))\dd t + \sum_{j=1}^{d}q_j f_{k_j}\dd W^j_t,
	\\
	w(0,w_0) 
	&= w_0. \notag
\end{alignat}
The state space is $\mathbb{L}^2$, the space of mean zero square integrable functions, with norm $\lVert\cdot\rVert$ and scalar product $\langle\cdot,\cdot\rangle$.
Furthermore, $\Delta$ is the Laplacian, $\mathcal{K}$ the inverse of the rotation $\nabla\wedge u=\partial_2 u_1-\partial_1 u_2$ in the space of divergence free vector fields, $\nabla\wedge(\mathcal{K}w)=w$ and $\nabla\cdot\mathcal{K}w=0$, $B(u,w)=-(u\cdot\nabla)w$ the Navier-Stokes nonlinearity, and $(W^j_t)_{j=1,\cdots,d}$ a $d$-dimensional Brownian motion.
The $q_j$ are nonvanishing real numbers, $q_j\in\mathbb{R}\setminus\left\{ 0 \right\}$,
and $f_k$ are the orthonormal eigenfunctions of $\Delta$ on $\mathbb{T}^2$,
\begin{equation}
	f_k(x)
	=
	\begin{cases}
		(2\pi^2)^{-1/2}\sin(k\cdot x), & k\in\mathbb{Z}^2_+, \\
		(2\pi^2)^{-1/2}\cos(k\cdot x), & \text{else}, \\
	\end{cases}
\end{equation}
where 
\begin{equation}
	\mathbb{Z}^2_+:=\left\{ k=(k_1,k_2)\in\mathbb{Z}^2\colon \text{either $k_2>0$, or $k_2=0$ and $k_1>0$} \right\}.
\end{equation}
We also define the Sobolev spaces of divergence-free, mean zero functions $\mathbb{H}^s$ with norm $\lVert \sum_{k\in\mathbb{Z}^2}w_k f_k\rVert_s:=\sqrt{\sum_{k\in\mathbb{Z}^2}(k_1^2+k_2^2)^s\lvert w_k\rvert^2}$, which is non-degenerate due to the mean zero condition (the term for $k=(0,0)$ vanishes).
We note, in particular, that
\begin{equation}
	-\langle\Delta w,w\rangle
	=
	\lVert w\rVert_1^2.
\end{equation}

Similarly as in \cite[Section~5.3]{HairerMattingly2008}, we introduce the weight $\psi_\eta(w):=\exp(\eta\lVert w\rVert^2)$ with some $\eta>0$ and consider the weighted space $\mathcal{B}^{\psi_\eta}(\mathbb{L}^2)$, given as the closure of the space of smooth, cylindrical functions $f=g(\langle\cdot,e_1\rangle,\dots,\langle\cdot,e_n\rangle)$, $g\in\mathrm{C}_b^{\infty}(\mathbb{R}^n)$, with respect to the norm
\begin{equation}
	\lVert f\rVert_{\psi_\eta}
	:=
	\sup_{w\in\mathbb{L}^2}\psi_{\eta}(w)^{-1}\lvert f(w)\rvert.
\end{equation}
For a more complete exposition of such spaces, we refer the reader to \cite{DoersekTeichmann2010}.
We remark that \cite[Section~5]{DoersekTeichmann2010} provides several conditions under which Markov semigroups of operators defined on such weighted spaces become strongly continuous, generalising the ideas in \cite[Section~5.3]{HairerMattingly2008}.
\begin{proposition}
	\label{prop:sns-strongcontsns}
	The Markov semigroup $(P_t)_{t\ge 0}$ defined through $P_t f(w_0):=\mathbb{E}[f(w(t,w_0))]$ is strongly continuous on $\mathcal{B}^{\psi_\eta}(\mathbb{L}^2)$ for $\eta>0$ small enough.
\end{proposition}
\begin{proof}
  This follows from \cite[Theorem~5.8]{DoersekTeichmann2010} and \cite[Theorem~A.3]{HairerMattingly2008}.
	A very similar result is proved in \cite[Theorem~5.10]{HairerMattingly2008}.
\qquad\end{proof}

Contrary to the approach used in \cite{NinomiyaVictoir2008,TanakaKohatsuHiga2009}, we are not able to split this problem into a part corresponding fully to the drift and another for the diffusion:
the process $y(t,w_0)_t:=w_0+\sum_{j=1}^{d}q_j W^j_t f_{k_j}$ corresponding to the diffusion does not satisfy $\mathbb{E}[\psi_{\eta}(y(t,w_0))]\le K\psi_{\eta}(w_0)$ with $K>0$ constant for $t$ small enough, which means that we cannot use standard Ninomiya-Victoir splittings.

Thus, we split up the equation differently.
For a given $\varepsilon\in(0,1)$, we introduce the deterministic vorticity equation,
\begin{alignat}{2}
	\label{eq:sns-splitvorticity}
	\frac{\dd}{\dd t} w^1(t,w_0)
	&=
	(1-\varepsilon)\nu\Delta w^1(t,w_0) + B(\mathcal{K}w^1(t,w_0),w^1(t,w_0)),
	\\
	w^1(0,w_0) 
	&= w_0, \notag
\end{alignat}
and a stochastic heat equation defining an Ornstein-Uhlenbeck process on $\mathbb{L}^2$,
\begin{equation}
	\label{eq:sns-splitou}
	\dd w^2(t,w_0)
	=
	\varepsilon\nu\Delta w^2(t,w_0)\dd t + \sum_{j=1}^{d}q_j f_{k_j}\dd W^j_t,
	\quad
	w^2(0,w_0) = w_0.
\end{equation}
Define by $P^1_tf(w_0):=\mathbb{E}[f(w^1(t,w_0))]$ and $P^2_tf(w_0):=\mathbb{E}[f(w^2(t,w_0))]$ the Markov semigroups corresponding to $w^1$ and $w^2$.
\begin{lemma}
	\label{lem:sns-strongcontdet}
	For $\eta>0$, $(P^1_t)_{t\ge 0}$ defines a strongly continuous semigroup on $\mathcal{B}^{\psi_{\eta}}(\mathbb{L}^2)$ with $\lVert P^1_t\rVert_{L(\mathcal{B}^{\psi_{\eta}}(\mathbb{L}^2))}\le1$.
\end{lemma}
\begin{proof}
  The strong continuity is obtained using \cite[Theorem~5.8]{DoersekTeichmann2010}.
	The necessary bounds are proved by applying \cite[Theorem~A.3]{HairerMattingly2008}; see also \cite[Theorem~5.10]{HairerMattingly2008}.

	The deterministic vorticity equations have $\mathbb{L}^2$-contractive dynamics, as
	\begin{alignat}{2}
		\lVert w^1(t,w_0)\rVert^2
		&=
		\lVert w_0\rVert^2 + \int_{0}^{t}\langle\varepsilon\nu\Delta w^1(s,w_0) 
		+B(\mathcal{K}w^1(s,w_0),w^1(s,w_0)),w^1(s,w_0)\rangle\dd s \notag \\
		&\le
		\lVert w_0\rVert^2,
	\end{alignat}
	which yields the norm bound.
	The proof is thus complete.
\qquad\end{proof}

The cumbersome proof of the following proposition is postponed to the appendix.
\begin{proposition}
	\label{prop:sns-ousupermart}
	If $\eta>0$ is small enough, there exists $\omega>0$ such that the process $t\mapsto\exp(-\omega t)\psi_\eta(w^2(t,w_0))$ is a positive supermartingale, i.e.
	\begin{equation}
	  \mathbb{E}[\psi_\eta(w^2(t,w_0)]\le\exp(\omega t)\psi_\eta(w^2(t,w_0)).
	\end{equation}
\end{proposition}
\begin{lemma}
	\label{lem:sns-strongcontou}
	For $\eta>0$ small enough, $(P^2_t)_{t\ge 0}$ is strongly continuous on $\mathcal{B}^{\psi_\eta}(\mathbb{L}^2)$ with bound $\lVert P^2_t\rVert_{L(\mathcal{B}^{\psi_{\eta}}(\mathbb{L}^2))}\le\exp(\omega t)$.
\end{lemma}
\begin{proof}
  Clear from Proposition~\ref{prop:sns-ousupermart} (see also \cite[Example~5.4]{DoersekTeichmann2010}).
\qquad\end{proof}

\section{Spectral Galerkin approximations for stochastic Navier-Stokes equations}
\label{sec:spectralgalerkin}
For the stochastic Navier-Stokes equations, we cannot argue directly as in \cite{DoersekTeichmann2010}:
there do not appear to be useful weight functions on spaces of more regular functions (such spaces are nevertheless invariant with respect to the dynamics of \eqref{eq:sns-conteq}; see \cite[Section~3.4]{Mattingly1998} in this regard).
We will therefore settle with a weaker result:
we shall prove that spectral Galerkin approximations using Fourier modes up to degree $N$ yield a convergent scheme, which can then be approximated by a splitting or a cubature scheme with $N$-dependent error bound.
As the $N$-dependence of the estimate is given explicitly, we can derive convergent schemes by choosing the time step size small enough in relation to $N$.

Consider therefore the spectral Galerkin approximation of \eqref{eq:sns-conteq},
\begin{subequations}
	\label{eq:snsspectralgalerkin}
	\begin{alignat}{2}
		\dd w_N(t,w_0)
		&=
		\nu\Delta w_N(t,w_0)\dd t 
		\\ &\phantom{=}
		+ \pi_N B(\mathcal{K}w_N(t,w_0),w_N(t,w_0))\dd t + \sum_{j=1}^{d}q_j f_{k_j}\dd W^j_t, \notag \\
		w_N(0,w_0) 
		&= \pi_N w_0, 
	\end{alignat}
\end{subequations}
see also \cite{EMattingly2001}, where $\pi_N\colon \mathbb{L}^2\to\mathbb{L}^2$ is the projection onto the space $\mathcal{H}_N$ of tensor products of trigonometric polynomials of degree $N$,
\begin{equation}
\mathcal{H}_N:=\lspan\left\{ f_k\colon \max_{i=1,2} \lvert k_i\rvert\le N \right\}, 
\end{equation}
and $N$ is assumed to be large enough so that $f_{k_j}\in\mathcal{H}_N$ for $j=1,\dots,d$.
Its split semigroups are given by
\begin{alignat}{2}{}
	\frac{\dd}{\dd t} w^1_N(t,w_0) 
	&= \pi_N B(\mathcal{K}w^1_N(t,w_0),w^1_N(t,w_0)), 
	&w^1_N(0,w_0) 
	&= w_0, 
	\qquad\text{and} \\
	\dd w^2_N(t,w_0)
	&= \nu\Delta w^2_N(t,w_0)\dd t + \sum_{j=1}^{d}q_j f_{k_j}\dd W^j_t,
	\quad
	&w^2_N(0,w_0) &= w_0.
\end{alignat}
The choice $\varepsilon=1$ made here is not admissible above:
in the space continuous setting, the results from \cite{HairerMattingly2008} do not allow us to apply \cite[Theorem~5.8]{DoersekTeichmann2010} to conclude that $P^1_t$ is strongly continuous for this choice. 
As $\mathcal{H}_N$ is finite-dimensional, however, we do not have to distinguish between different topologies, and it follows that the Markov semigroups $P^N_t$, $P^{N,1}_t$ and $P^{N,2}_t$ of $w_N$, $w^1_N$ and $w^2_N$ are strongly continuous on $\mathcal{B}^{\psi_\eta}(\mathcal{H}_N)$ if $\eta>0$ is small enough.
In case that a solver for deterministic Navier-Stokes equations is available, it is also possible to use $\varepsilon<1$ here (the case $\varepsilon=1$ corresponds to splitting up into a deterministic Euler equation).

We now estimate the error of the spectral Galerkin approximation.
\begin{proposition}
	\label{prop:spectralgalerkinest}
	For any $\alpha>0$, $w_0\in\mathbb{L}^2$ and $t>0$,
	\begin{alignat}{2}
		&\lVert w(t,w_0)-w_N(t,w_0)\rVert^2
		\le
		CN^{-1}\lVert w(t,w_0)\rVert_1^2 \notag \\
		&\quad
		+ C_\alpha N^{-1} \exp\left( C_{\alpha}t + \frac{\alpha}{2}\int_{0}^{t}\lVert w(\sigma,w_0)\rVert_1^2\dd\sigma \right) \int_{0}^{t}\lVert w(s,w_0)\rVert_1^4 \dd s.
	\end{alignat}
\end{proposition}
\begin{proof}
	Let $e_N(t):=\pi_N w(t,w_0)-w_N(t,w_0)\in\mathcal{H}_N$ and $\eta_N(t):=w(t,w_0)-\pi_N w(t,w_0)$.
	Then,
	\begin{alignat}{2}{}
		\dd e_N(t)
		&=\nu\Delta e_N(t) + \pi_N\left( B(\mathcal{K}w_N(t,w_0),e_N(t)) + B(\mathcal{K}e_N(t),\pi_Nw(t,w_0)) \right)\dd t \notag \\
		&\phantom{=}+ \pi_N\left( B(\mathcal{K}\pi_N w(t,w_0),\eta_N(t)) + B(\mathcal{K}\eta_N(t),w(t,w_0)) \right)\dd t.
	\end{alignat}
	It results that
	\begin{alignat}{2}{}
		\frac{1}{2}\frac{\dd}{\dd t}\lVert e_N(t)\rVert^2
		&= -\nu\lVert e_N(t)\rVert_1^2\dd t + \langle B(\mathcal{K}e_N(t),\pi_N w(t,w_0)),e_N(t)\rangle \notag \\
		&\phantom{=}+ \langle B(\mathcal{K}\pi_N w(t,w_0),\eta_N(t)) + B(\mathcal{K}\eta_N(t),w(t,w_0)),e_N(t)\rangle.
	\end{alignat}
	We now proceed similarly as in \cite[Proof of Lemma 4.10, point 3]{HairerMattingly2006}.
	For any $\delta>0$, we estimate 
	\begin{equation}
		\lvert\langle B(\mathcal{K}h,w),\zeta\rangle\rvert
		\le
		\delta\lVert\zeta\rVert_1^2 + \frac{C}{4\alpha^2\delta}\lVert\zeta\rVert^2 + \frac{\alpha}{4}\lVert w\rVert_1^2\lVert h\rVert^2.
	\end{equation}
	This yields
	\begin{alignat}{2}{}
		\lvert\langle B(\mathcal{K}e_N(t),\pi_N w(t,w_0)),e_N(t)\rangle\rvert
		&\le \delta\lVert e_N(t)\rVert_1^2 + \frac{C}{4\alpha^2\delta}\lVert e_N(t)\rVert^2 
		\notag \\ &\phantom{\le}
		+ \frac{\alpha}{4}\lVert \pi_N w(t,w_0)\rVert_1^2\lVert e_N(t)\rVert^2
		\quad\text{and} \\
		\lvert\langle B(\mathcal{K}\eta_N(t),w(t,w_0)),e_N(t)\rangle\rvert
		&\le \delta\lVert e_N(t)\rVert_1^2 + \frac{C}{4\alpha^2\delta}\lVert e_N(t)\rVert^2 
		\notag \\ &\phantom{\le}
		+ \frac{\alpha}{4}\lVert w(t,w_0)\rVert_1^2\lVert \eta_N(t)\rVert^2.
	\end{alignat}
	For the final term, we apply
	\begin{equation}
		\lvert \langle B(\mathcal{K}h,w),\zeta\rangle\rvert
		\le
		\delta\lVert \zeta\rVert_1^2 + \frac{C}{4\delta}\lVert h\rVert_1^2\lVert w\rVert^2,
	\end{equation}
	which shows
	\begin{alignat}{2}{}
		\lvert \langle B(\mathcal{K}\pi_N w(t,w_0),\eta_N(t)),e_N(t)\rangle\rvert
		&\le \delta\lVert e_N(t)\rVert_1^2 
		+ \frac{C}{4\delta}\lVert \pi_N w(t,w_0)\rVert_1^2\lVert \eta_N(t)\rVert^2.
	\end{alignat}
	Choosing $\delta=\frac{\nu}{6}$ and combining the above estimates yields
	\begin{alignat}{2}{}
		\frac{1}{2} \frac{\dd}{\dd t} \lVert e_N(t)\rVert^2
		&\le -\frac{\nu}{2}\lVert e_N(t)\rVert_1^2 + \frac{3C}{\alpha^2\nu}\lVert e_N(t)\rVert^2 + \frac{\alpha}{4}\lVert \pi_Nw(t,w_0)\rVert_1^2\lVert e_N(t)\rVert^2 \notag \\
		&\phantom{\le}+ \left( \frac{\alpha}{4}\lVert w(t,w_0)\rVert_1^2 + \frac{3C}{\nu}\lVert\pi_N w(t,w_0)\rVert_1^2 \right)\lVert \eta_N(t)\rVert^2.
	\end{alignat}
	Using $\lVert\pi_N w\rVert_1\le\lVert w\rVert_1$, we obtain
	\begin{equation}
		\frac{1}{2} \frac{\dd}{\dd t} \lVert e_N(t)\rVert^2
		\le
		\left( C_{\alpha} + \frac{\alpha}{2}\lVert w(t,w_0)\rVert_1^2\right) \frac{1}{2}\lVert e_N\rVert^2 + C_\alpha\lVert w(t,w_0) \rVert_1^2 \lVert \eta_N(t)\rVert^2.
	\end{equation}
	An application of Gronwall's inequality yields, as $e_N(0)=0$,
	\begin{alignat}{2}{}
		\frac{1}{2}\lVert e_N(t)\rVert^2
		&\le
		\int_{0}^{t}C_\alpha\lVert w(s,w_0)\rVert_1^2\lVert\eta_N(s)\rVert^2 
		\times \notag\\ &\phantom{\le}\times 
		\exp\left( C_{\alpha}(t-s) + \frac{\alpha}{2}\int_{s}^{t}\lVert w(\sigma,w_0)\rVert_1^2\dd\sigma \right) \dd s.
	\end{alignat}
	As $\lVert w-\pi_N w\rVert\le CN^{-1}\lVert w\rVert_1$, we see that $\lVert\eta_N(t)\rVert\le CN^{-1}\lVert w(t,w_0)\rVert_1$, which yields
	\begin{alignat}{2}{}
		\frac{1}{2}\lVert e_N(t)\rVert^2 
		&\le
		C_\alpha N^{-1}\int_{0}^{t}\lVert w(s,w_0)\rVert_1^4\exp\left( C_{\alpha}(t-s) + \frac{\alpha}{2}\int_{s}^{t}\lVert w(\sigma,w_0)\rVert_1^2\dd\sigma \right) \dd s \notag \\
		&\le
		C_\alpha N^{-1} \exp\left( C_{\alpha}t + \frac{\alpha}{2}\int_{0}^{t}\lVert w(\sigma,w_0)\rVert_1^2\dd\sigma \right) \int_{0}^{t}\lVert w(s,w_0)\rVert_1^4 \dd s.
	\end{alignat}
	The result follows due to $\lVert w(t,w_0)-w_N(t,w_0)\rVert\le\lVert e_N(t)\rVert + CN^{-1}\lVert w(t,w_0)\rVert_1$.
\qquad\end{proof}
\begin{corollary}
	\label{cor:spectralgalerkinestL2}
	For any $w_0\in\mathbb{H}^1$ and $T\ge 0$, there exists a constant $C=C_{w_0,T}>0$ such that for any $t\in[0,T]$,
	\begin{equation}
		\mathbb{E}\left[ \lVert w(t,w_0)-w_N(t,w_0) \rVert^2 \right]
		\le
		C N^{-1}.
	\end{equation}
\end{corollary}
\begin{proof}
	From Proposition~\ref{prop:spectralgalerkinest} and an application of the Cauchy-Schwarz inequality, we see that we need to prove 
	\begin{alignat}{2}
		\mathbb{E}[\lVert w(t,w_0)\rVert_1^2] 
		&+ \mathbb{E}\left[\exp\left(\alpha\int_{0}^{t}\lVert w(\sigma,w_0)\rVert_1^2\dd\sigma\right)\right] 
		\notag \\ &
		+ \mathbb{E}\left[\left(\int_{0}^{t}\lVert w(s,w_0)\rVert_1^4\dd s\right)^2\right]
		\le K
	\end{alignat}
	for all $t\in[0,T]$ with some $K=K_{t,w_0}>0$.
	For the first and third term, this follows from \cite[Theorem~3.7]{Mattingly1998}, and for the second, from \cite[Lemma~4.10]{HairerMattingly2006}.
\qquad\end{proof}
\begin{remark}
	Actually, it seems quite plausible here that the assumption $w_0\in\mathbb{H}^1$ is too strong.
	Indeed, the results in \cite{Mattingly2002} show that if $w_0\in\mathbb{L}^2$, then $w(t,w_0)\in\mathbb{H}^s$ for all $s>0$ for subsequent times, and \cite[Lemma~A.3]{MattinglyPardoux2006} gives some quantitative estimates.
	It remains unclear to us however how this can be used to prove an estimate for $\mathbb{E}\left[ \left( \int_{0}^{t}\lVert w(s,w_0)\rVert_1^4\dd s \right)^2 \right]$.
\end{remark}

The estimate from Corollary~\ref{cor:spectralgalerkinestL2} allows us to estimate the pointwise approximation error of the weak approximation of the stochastic Navier-Stokes equation by the spectral Galerkin scheme.
\begin{theorem}
	\label{thm:markovgalerkinest}
	Assume $\varphi\in\mathcal{B}^{\psi_{\eta}}(\mathbb{L}^2)\cap\mathrm{C}^1(\mathbb{L}^2)$ with 
	\begin{equation}
		\label{eq:phicondspectralapprox}
		C_{\varphi}:=\sup_{w\in\mathbb{L}^2}\psi_{\tilde{\eta}}(w)^{-1}\lVert D\varphi(w)\rVert
		<\infty
	\end{equation}
	for some $\tilde{\eta}\in[0,\eta/2]$.
	Then, for $w\in\mathbb{H}^1$ and $T\ge 0$, there exists a constant $C=C_{w,T,\varphi}$ such that for all $t\in[0,T]$,
	\begin{equation}
		\lvert P_t\varphi(w) - P^N_t(\varphi|_{\mathcal{H}_N})(w)\rvert
		\le C N^{-1}.
	\end{equation}
\end{theorem}
\begin{proof}
	By the fundamental theorem of calculus,
	\begin{alignat}{2}{}
		&\lvert\varphi(w(t,w_0))-\varphi(w_N(t,w_0))\rvert
		\\
		\notag
		&\quad
		\le \int_{0}^{1}\lVert D\varphi(\theta w(t,w_0) + (1-\theta)w_N(t,w_0))\rVert \cdot \lVert w(t,w_0)-w_N(t,w_0))\rVert\dd\theta.
	\end{alignat}
	The assumption on $\varphi$ together with the convexity of $w\mapsto\exp(\tilde{\eta}\lVert w\rVert^2)$ yields
	\begin{alignat}{2}
		&\lVert D\varphi(\theta w(t,w_0) + (1-\theta)w_N(t,w_0))\rVert
		\\
		\notag 
		&\qquad\le
		C_{\varphi}\left( \exp(\tilde{\eta}\lVert w(t,w_0)\rVert^2) + \exp(\tilde{\eta}\lVert w_N(t,w_0)\rVert^2) \right).
	\end{alignat}
	Therefore, the Cauchy-Schwarz inequality implies
	\begin{alignat}{2}
		\lvert P_t \varphi(w) &- P^N_t(\varphi|_{\mathcal{H}_N}(w) \rvert
		\le C_{\varphi}\mathbb{E}\left[ \lVert w(t,w_0)-w_N(t,w_0)\rVert^2 \right]^{1/2} \times \\
		&\phantom{\le}\times \left( \mathbb{E}[\exp(2\tilde{\eta}\lVert w(t,w_0)\rVert^2)]^{1/2} + \mathbb{E}[\exp(2\tilde{\eta}\lVert w_N(t,w_0)\rVert^2)]^{1/2} \right) \notag.
	\end{alignat}
	Note that the estimate in \cite[Lemma 4.10, 1.]{HairerMattingly2006} also holds true for $w_N(t,w_0)$ instead of $w(t,w_0)$.
	Therefore, Corollary~\ref{cor:spectralgalerkinestL2} proves the claimed estimate.
\qquad\end{proof}

In the discrete setting, it is easy to analyse the differential operators corresponding to the split semigroups.
For $k\ge 0$, we define $\mathcal{B}^{\psi_{\eta}}_k(\mathcal{H}_N)$ as the closure of $\mathrm{C}_b^{\infty}(\mathcal{H}_N)$ with respect to the norm
\begin{equation}
	\lVert f\rVert_{\psi_{\eta},k}
	:=
	\lVert f\rVert_{\psi_{\eta}}
	+
	\sum_{j=1}^{k}\lvert f\rvert_{\psi_{\eta},j},
\end{equation}
where the seminorms $\lvert\cdot\rvert_{\psi_{\eta},j}$, $j=1,\dots,k$, are given by
\begin{equation}
	\lvert f\rvert_{\psi_{\eta},j}
	:=
	\sup_{w\in\mathcal{H}_N}\psi_{\eta}(w)^{-1}\lVert D^j f(w)\rVert_{L( (\mathcal{H}_N)^{\bigotimes j};\mathbb{R})}.
\end{equation}
We denote by $\mathcal{G}^N_j$ with domain $\dom\mathcal{G}^N_j$ the infinitesimal generator of $(P^{N,j}_t)_{t\ge 0}$, $j=1,2$, and by $\mathcal{G}^N$ with domain $\dom\mathcal{G}^N$ the infinitesimal generator of $(P^N_t)_{t\ge 0}$.
\begin{lemma}
	\label{lem:estimateGN}
	For any $\varepsilon>0$, 
	\begin{equation}
		\mathcal{B}^{\psi_{\tilde{\eta}}}_2(\mathcal{H}_N)\subset\dom\mathcal{G}^N\cap\dom\mathcal{G}^N_1\cap\dom\mathcal{G}^N_2.
	\end{equation}
	For $k\ge 0$, $\mathcal{G}^N$, $\mathcal{G}^N_j\colon \mathcal{B}^{\psi_{\tilde{\eta}}}_{k+2}(\mathcal{H}_N)\to\mathcal{B}^{\psi_{\tilde{\eta}+\varepsilon}}_k(\mathcal{H}_N)$, $j=1,2$, are continuous operators, and
	\begin{equation}
		\lVert \mathcal{G}^N\rVert_{L(\mathcal{B}^{\psi_{\tilde{\eta}}}_{k+2}(\mathcal{H}_N);\mathcal{B}^{\psi_{\tilde{\eta}+\varepsilon}}_{k}(\mathcal{H}_N))}
		+
		\lVert \mathcal{G}^N_j\rVert_{L(\mathcal{B}^{\psi_{\tilde{\eta}}}_{k+2}(\mathcal{H}_N);\mathcal{B}^{\psi_{\tilde{\eta}+\varepsilon}}_{k}(\mathcal{H}_N))}
		\le
		CN^2,
		\quad j=1,2.
	\end{equation}
	Furthermore,
	\begin{equation}
		\label{eq:equalitygenerators}
		\mathcal{G}^N\varphi=\mathcal{G}^N_1\varphi+\mathcal{G}^N_2\varphi
		\quad\text{for all $\varphi\in\mathcal{B}^{\psi_{\tilde{\eta}}}_2(\mathcal{H}^N)$}.
	\end{equation}
\end{lemma}
\begin{proof}
	For $\varphi\in\mathcal{B}^{\psi_{\tilde{\eta}}}_{k+2}(\mathcal{H}_N)$, we see by the fundamental theorem of calculus and the estimates in \cite[Appendix]{HairerMattingly2008} that with $\alpha>0$, 
	\begin{alignat}{2}
		\lvert \mathcal{G}^N_1\varphi(w)\rvert
		&=
		\lvert D\varphi(w)\left( \pi_NB(\mathcal{K}w,w) \right)\rvert 
		\le \lVert D\varphi(w)\rVert\cdot\left( N^{1+\alpha}\lVert w\rVert^2 \right) \notag \\
		&\le
		C N^2\exp(\varepsilon\lVert w\rVert^2)\lVert D\varphi(w)\rVert,
	\end{alignat}
	and similarly, by It\^o's formula,
	\begin{alignat}{2}{}
		\lvert \mathcal{G}^N_2\varphi(w)\rvert
		&=
		\lvert D\varphi(w)\nu\Delta w + \frac{1}{2}\sum_{j=1}^{d}D^2\varphi(w)(q_jf_{k_j},q_jf_{k_j})\rvert 
		\\ &
		\le \lVert D\varphi(w)\rVert\cdot\nu N^2\lVert w\rVert + C\lVert D^2\varphi(w)\rVert \notag \\
		&\le
		CN^2\exp(\varepsilon\lVert w\rVert^2)\left( \lVert D\varphi(w)\rVert + \lVert D^2\varphi(w)\rVert \right). \notag
	\end{alignat}
	The result for $\mathcal{G}^N$ is proved in a similar manner.
	The equality \eqref{eq:equalitygenerators} is a consequence of It\^o's formula if $\varphi\in\mathrm{C}_b^{\infty}(\mathcal{H}_N)$, and a density argument proves it for the general case.
\qquad\end{proof}

\section{Rates of convergence}
\label{sec:rateconv}
We are now in the situation to prove estimates for the convergence of both splitting schemes and cubature methods.
\subsection{Splitting methods}

\begin{lemma}
	For all $k\ge 0$, $P^N_t\mathcal{B}^{\psi_{\tilde{\eta}}}_k(\mathcal{H}_N)\subset\mathcal{B}^{\psi_{\tilde{\eta}}}_k(\mathcal{H}_N)$ and $\sup_{t\in[0,T]}\lVert P^N_t \varphi\rVert_{\psi_{\tilde{\eta}},k}\le K_T\lVert \varphi\rVert_{\psi_{\tilde{\eta}},k}$ with some constant $K_T$ independent of $\varphi$.
\end{lemma}
\begin{proof}
	This is proved using similar estimates as those given in \cite[Lemma~4.10, 1. and 3.]{HairerMattingly2006}.
\qquad\end{proof}

Using Lemma~\ref{lem:estimateGN}, the method of \cite{HansenOstermann2009} yields the following convergence estimate.
\begin{theorem}
	\label{thm:discretenvest}
	Let $Q^N_{(\Delta t)}:=P^{N,1}_{\Delta t/2}P^{N,2}_{\Delta t}P^{N,1}_{\Delta t/2}$ denote the Strang splitting approximation of $P^N_{\Delta t}$ using $P^{N,1}_{\Delta t}$ and $P^{N,2}_{\Delta t}$.
	For any $\tilde{\eta}<\eta/2$, there exists $C=C_{T,\tilde{\eta}}>0$ such that for all $\varphi\in\mathcal{B}^{\psi_{\tilde{\eta}}}_6(\mathcal{H}_N)$ and $n\in\mathbb{N}$,
	\begin{equation}
		\lVert P^N_{T}\varphi-(Q^N_{(T/n)})^n\varphi \rVert_{\psi_\eta}
		\le
		C_T N^6 n^{-2}\lVert\varphi\rVert_{\psi_{\tilde{\eta}},6}.
	\end{equation}
\end{theorem}

Note that if $\varphi\in\mathrm{C}^{6}(\mathbb{L}^2)$ is such that for some $\tilde{\eta}<\eta$,
\begin{equation}
	\label{eq:condphiconv}
	\sup_{w\in\mathbb{L}^2}\psi_{\tilde{\eta}}(w)^{-1}\lVert D^j\varphi(w)\rVert_{L( (\mathbb{L}^2)^{\bigotimes j};\mathbb{R} )}<\infty
	\quad\text{for $j=0,\dots,6$},
\end{equation}
then $\varphi\vert_{\mathcal{H}_N}\in\mathcal{B}^{\psi_{\eta}}_{6}(\mathcal{H}_N)$ for all $N\in\mathbb{N}$.
Furthermore, \eqref{eq:condphiconv} with $\tilde{\eta}<\eta/2$ implies \eqref{eq:phicondspectralapprox}.
Thus, we obtain the following result.
\begin{corollary}
	\label{cor:splitting-convresultfull}
	Assume that $\varphi$ satisfies \eqref{eq:condphiconv} with $\tilde{\eta}<\eta/2$.
	For any $T>0$ and $w_0\in\mathbb{H}^1$, there exists $C=C_{w_0,T,\varphi}>0$ such that for all $n\in\mathbb{N}$
	\begin{equation}
		\lvert P_{T}\varphi(w_0) - (Q^N_{(T/n)})^n\varphi|_{\mathcal{H}_N}(w_0)\rvert
		\le C\left( N^{-1} + N^6 n^{-2} \right).
	\end{equation}
\end{corollary}
\begin{proof}
	The combination of Theorem~\ref{thm:markovgalerkinest} and Theorem~\ref{thm:discretenvest} allows us to conclude the desired estimate.
\qquad\end{proof}
\begin{remark}
  We see here an important advantage of the second order splitting in comparison to a possible first order splitting.
	There, in the second term, the instability would be of the order $N^4$, but the convergence would only be of first order, $n^{-1}$.
	Therefore, we can choose $n^{-2}$ significantly larger here while still obtaining a stable method.
Nevertheless, we have to stress that the given error estimate is far from what we would expect to obtain, see also the numerical results in Section~\ref{sec:numerics}.
\end{remark}

\subsection{Cubature methods}
We give a short overview of cubature methods.
For more detailed accounts, see \cite{LyonsVictoir2004,GyurkoLyons2011,CrisanGhazali2007}.

Fix $M\in\mathbb{N}$.
A set of paths $\omega_i=(\omega_i^j)_{j=0}^{d}\colon[0,1]\to\mathbb{R}^{d+1}$, $\omega_i^{0}(s)=s$, $i=1,\dots,M$, of bounded variation with $\omega_i(0)=0$ and weights $\lambda_i>0$, $i=1,\dots,M$, is called \emph{cubature formula on Wiener space of order $m$} if
\begin{alignat}{2}{}
	\mathbb{E}&[\idotsint_{0<t_1<\dots<t_k<1}\circ\dd W^{j_1}_{t_1}\dots\circ\dd W^{j_k}_{t_k}] \notag \\
	&=
	\sum_{i=1}^{M}\lambda_i\idotsint_{0<t_1<\dots<t_k<1}\dd \omega_i^{j_1}(t_1)\dots\dd\omega_i^{j_k}(t_k).
\end{alignat}
Here, $(j_1,\dots,j_k)\in\left\{ 0,\dots,d \right\}^{k}$ runs through all multiindices satisfying
\begin{equation}
	k+\#\left\{ i\colon j_i=0 \right\}\le m.
\end{equation}
This means that the convex combination of iterated integrals along the cubature paths up to order $m$ equals the expected value of the corresponding iterated Stratonovich integrals along $d$-dimensional Brownian motion.
To scale the paths to an interval $[0,\Delta t]$, we define $\omega_i^{(\Delta t)}\colon[0,\Delta t]\to\mathbb{R}^{d+1}$ by $\omega_i^{(\Delta t),0}(t):=t$ and $\omega_i^{(\Delta t),j}(t):=\sqrt{\Delta t}\omega_i^{j}\left( \frac{t}{\Delta t} \right)$.
The cubature approximations of the spectral Galerkin discretisation of the stochastic Navier-Stokes equations over a time step of size $\Delta t$ are then given by
\begin{alignat}{2}{}
	\dd w_N(s,w_0;\omega^{(\Delta t)}_i)
	=
	&\left( \nu\Delta w_N(s,w_0;\omega^{(\Delta t)}_i) + \pi_N B(\mathcal{K}w_N(s,w_0;\omega^{(\Delta t)}_i)) \right)\dd s \notag \\
	&+ \sum_{j=1}^{d}q_j f_{k_j}\dd\omega^{(\Delta t), j}_i(s).
\end{alignat}
Here, we apply that the noise is purely additive, entailing that the It\^o and Stratonovich integrals of the noise terms coincide.
The cubature approximation of the Markov semigroup $P^N_{\Delta t}$ reads 
\begin{equation}
  Q^{N}_{(\Delta t)}f(w_0)
  :=
  \sum_{i=1}^{M}\lambda_i f(w_N(\Delta t,w_0;\omega^{(\Delta t)}_i)).
\end{equation}

To prove stability of the cubature approximation,
we require that the quadrature formula induced by the cubature scheme is symmetric, i.e., for all $i=1,\dots,M$, there exists a unique $i'\in\left\{ 1,\dots,M \right\}$ such that $\lambda_i=\lambda_{i'}$ and $\omega^{j}_i(\Delta t)=-\omega^{j}_{i'}(\Delta t)$ for $j=1,\dots,d$.
This induces a corresponding symmetry for $\omega^{(\Delta t)}$.
Many known cubature formulas satisfy such a property, consider e.g.~the paths given in \cite{LyonsVictoir2004}.
Moreover, given an arbitrary cubature formula, it is easy to construct a symmetric one from it by adding the reflected paths.

Our use of this assumption is to prove an estimate for the moment generating function of the cubature paths at $\Delta t$.
\begin{lemma}
	\label{lem:mgfcub}
	Assume that the quadrature formula induced by the cubature scheme is symmetric.
	Then, for all continuous $f\colon\mathbb{R}^d\to\mathbb{R}$,
	\begin{alignat}{2}{}
		\sum_{i=1}^{M}\lambda_i f(\omega^{(\Delta t),1}_i(\Delta t),\dots&,\omega^{(\Delta t),d}_i(\Delta t))
		\notag \\
		=
		\frac{1}{2}\sum_{i=1}^{M}\lambda_i \Bigl( 
		&f(\omega^{(\Delta t),1}_i(\Delta t),\dots,\omega^{(\Delta t),d}_i(\Delta t))
		\notag \\
		&+
		f(-\omega^{(\Delta t),1}_i(\Delta t),\dots,-\omega^{(\Delta t),d}_i(\Delta t))
		\Bigr).
	\end{alignat}
	In particular, $\sum_{i=1}^{M}\lambda_if(\omega^{(\Delta t),1}_i(\Delta t),\dots,\omega^{(\Delta t),d}_i(\Delta t))=0$ if $f$ is odd.

	This implies
	\begin{equation}
		\sum_{i=1}^{M}\lambda_i\exp\left( \sum_{j=1}^{d}u_j\omega^{(\Delta t),j}_i(\Delta t) \right)
		\le\exp\left( \frac{C}{2}\Delta t\sum_{j=1}^{d}u_j^2 \right).
	\end{equation}
\end{lemma}
\begin{proof}
	The first two claims are clear.
	For the estimate of the moment generating function, note that, 
	as $\lvert\omega^{(\Delta t),j}_i(\Delta t)\rvert\le C\sqrt{\Delta t}$ and $(2\ell)!\le 2^{\ell}\ell!$,
	\begin{alignat}{2}{}
		\sum_{i=1}^{M}\lambda_i&\exp\Bigl( \sum_{j=1}^{d}u_j\omega^{(\Delta t),j}_i(\Delta t) \Bigr)
		=
		\sum_{k=0}^{\infty}\frac{1}{k!}\sum_{i=1}^{M}\lambda_i\Bigl( \sum_{j=1}^{d}u_j\omega^{(\Delta t),j}_i(\Delta t) \Bigr)^k
		 \\
		&=
		\sum_{\ell=0}^{\infty}\frac{1}{(2\ell)!}\sum_{i=1}^{M}\lambda_i\Bigl( \sum_{j=1}^{d}u_j\omega^{(\Delta t),j}_i(\Delta t) \Bigr)^{2\ell}
		\le \exp\Bigl( \frac{C}{2}\Delta t \sum_{j=1}^{d}u_j^2 \Bigr),
		\notag
	\end{alignat}
	which proves the given estimate.
\end{proof}
\begin{theorem}
	Assume that the quadrature formula induced by the cubature scheme is symmetric.
	Then, there exist $\eta_0>0$ and $\varepsilon>0$, depending only on the given problem data, but not on the discretisation parameter $N$, such that with a constant $C>0$ independent of $\Delta t$ and $N$,
	\begin{alignat}{2}{}
    \lVert Q^N_{(\Delta t)}f \rVert_{\psi_{\eta}}
    \le
    &\exp(C\Delta t)\lVert f\rVert_{\psi_{\eta}}
		\\ \notag 
		&\text{for $\Delta t\in(0,\varepsilon]$, $\eta\in(0,\eta_0]$, and $f\in\mathcal{B}^{\psi_{\eta}}(\mathcal{H}_N)$}.
	\end{alignat}
\end{theorem}
\begin{proof}
	Set $w_N(s):=w_N(s,w_0;\omega^{(\Delta t)}_i)$ and $V^N(w_N):=\nu\Delta w_N+\pi_N B(\mathcal{K}w_N,w_N)$.
	For every $\alpha\ge 0$,
	\begin{alignat}{2}{}
		\exp(\alpha s)\lVert w_N(s)\rVert^2& - \lVert w_N(0)\rVert^2
		=
		\int_{0}^{s}\exp(\alpha r)( \alpha\lVert w_N(r)\rVert^2 + 2\langle V^N(w_N(r)),w_N(r) \rangle )\dd r 
		\notag \\ & 
		+ 2\sum_{j=1}^{d}\int_{0}^{s}\exp(\alpha r)\langle q_j f_{k_j},w_N(r)\rangle\dd\omega^{(\Delta t),j}_i(r).
	\end{alignat}
	Applying Fubini's theorem and integration by parts to 
	\begin{alignat}{2}{}
		\int_{\sigma}^{\tau}\exp(\alpha r)\dd\omega^{(\Delta t),j}_i(r)
		&=\exp(\alpha\tau)\omega^{(\Delta t),j}_i(\tau) - \exp(\alpha\sigma)\omega^{(\Delta t),j}_i(\sigma) 
		\notag \\ &\phantom{=}
		- \alpha\int_{\sigma}^{\tau}\omega^{(\Delta t),j}_i(r)\exp(\alpha r)\dd r,
	\end{alignat}
	we obtain that
	\begin{alignat}{2}{}
		\int_{0}^{s}&\exp(\alpha r)\langle q_j f_{k_j},w_N(r)\rangle\dd\omega^{(\Delta t),j}_i(r)
		=
		\langle q_j f_{k_j},w_N(0)\rangle\int_{0}^{s}\exp(\alpha r)\dd\omega^{(\Delta t),j}_i(r)
		\notag \\ & \phantom{=}
		+ \int_{0}^{s}\exp(\alpha r)\int_{0}^{r}\langle q_j f_{k_j},V^N(w_N(q))\rangle\dd q\dd\omega^{(\Delta t),j}_i(r)
		\notag \\ & \phantom{=}
		+ \sum_{i=1}^{d}\int_{0}^{s}\exp(\alpha r)\int_{0}^{r}\langle q_{j} f_{k_{j}},q_{i}f_{k_{i}}\rangle\dd\omega^{(\Delta t),i}_i(q)\dd\omega^{(\Delta t),j)}_i(r)
		\notag \\ 
		&\le
		\langle q_j f_{k_j},w_N(0)\rangle \exp(\alpha s)\omega^{(\Delta t),j}_i(s) + C\exp(\alpha s)\lVert w_N(0)\rVert^2\Delta t + C\exp(\alpha s)s^2
		\notag \\ & \phantom{\le}
		+C\sqrt{\Delta t}\int_{0}^{s}\exp(\alpha q)\lVert V^N(w_N(q))\rVert_{-3}\dd q
		+ C\exp(\alpha s)s.
	\end{alignat}
	This yields, as $\langle w_N,V^N(w_N)\rangle=-\nu\lVert w_N\rVert_1^2$ and $\lVert V^N(w_N)\rVert_{-3}\le \lVert w_N\rVert + C\lVert w_N\rVert^2$, 
	\begin{alignat}{2}{}
		&\lVert w_N(\Delta t)\rVert^2
		\le
		\bigl( \exp(-\alpha\Delta t) + C\Delta t \bigr)\lVert w_N(0)\rVert^2
		\\ &
		+ 2\sum_{j=1}^{d}\langle q_j f_{k_j},w_N(0)\rangle \omega^{(\Delta t),j}_i(\Delta t)
		+ C\Delta t
		+ C(\Delta t)^2
		\notag \\ &
		+ \int_{0}^{\Delta t}\exp(\alpha(q-\Delta t))\left( (\alpha+C\sqrt{\Delta t})\lVert w_N(q)\rVert^2 + C\sqrt{\Delta t}\lVert w_N(q)\rVert - 2\nu\lVert w_N(q)\rVert_1^2 \right)\dd q.
		\notag 
	\end{alignat}
	Fix $\alpha=\nu$.
	As $\lVert w_N\rVert_1\ge\lVert w_N\rVert$, we can choose $\varepsilon>0$ such that for $\Delta t\in(0,\varepsilon]$,
	\begin{equation}
		\nu\lVert w_N\rVert^2 + C\sqrt{\Delta t}(\lVert w_N\rVert+\lVert w_N\rVert^2) - 2\nu\lVert w_N\rVert_1^2
		\le C\Delta t.
	\end{equation}
	By Lemma~\ref{lem:mgfcub},
	\begin{alignat}{2}{}
		\sum_{i=1}^{M}\lambda_i\exp\Bigl( 2\eta\sum_{j=1}^{d}\langle q_j f_{k_j},w_N(0)\rangle\omega^{(\Delta t),j}_i(\Delta t) \Bigr)
		&\le
		\exp\Bigl(\eta^2C\Delta t\sum_{j=1}^{d}\langle q_j f_{k_j},w_N(0)\rangle^2\Bigr)
		\notag \\
		&\le
		\exp(\eta^2C\Delta t\lVert w_N(0)\rVert^2).
	\end{alignat}
	Hence, for $\Delta t\in(0,\varepsilon]$,
	\begin{alignat}{2}{}
		\sum_{i=1}^{M}&\lambda_i\exp(\eta\lVert w_N(\Delta t,w_0;\omega^{(\Delta t),j}_i)\rVert^2)
		\\ \notag
		&\le
		\exp\Bigl( C\Delta t + \eta\lVert w_N(0)\rVert^2\bigl( \exp(-\nu\Delta t) + \eta C\Delta t \bigr) \Bigr).
	\end{alignat}
	Choosing $\eta_0>0$ small enough, we see that
	\begin{equation}
		\exp(-\nu\Delta t)+\eta C\Delta t\le 1
		\quad\text{for $\Delta t\in(0,\varepsilon]$ and $\eta\in(0,\eta_0]$}.
	\end{equation}
	The claim is thus proved.
\qquad\end{proof}
\begin{remark}
	It is clear from the proof that a corresponding result can also be shown in the space continuous case.
	As remarked before in the context of the splitting scheme, however, we are not able to derive rates of convergence in this setting, which is why we focus on the space discrete case.
\end{remark}

As it is straightforward to obtain an asymptotic expansion of $Q^{N}_{(\Delta t)}$ by the fundamental theorem of calculus (see \cite{LyonsVictoir2004,CrisanGhazali2007,BayerTeichmann2008}), we have the following result.
\begin{theorem}
	Fix $\eta>0$ small enough.
	Given $T>0$ and $\tilde{\eta}<\eta/2$ and assuming that $m$ is odd, there exist constants $\varepsilon>0$ and $C=C_{T,\tilde{\eta}}>0$ such that for all $\varphi\in\mathcal{B}^{\psi_{\tilde{\eta}}}_{6}(\mathcal{H}_N)$ and $n\in\mathbb{N}$ with $T/n<\varepsilon$,
	\begin{equation}
		\lVert P_{T}\varphi - (Q^N_{(T/n)})^n \varphi\rVert_{\psi_{\eta}}
		\le
		CN^{2\frac{m+1}{2}}(\Delta t)^{\frac{m-1}{2}}\lVert\varphi\rVert_{\psi_{\tilde{\eta}},6}.
	\end{equation}
\end{theorem}
The following result is a version of Corollary~\ref{cor:splitting-convresultfull} for cubature approximations.
\begin{corollary}
	Suppose $m$ odd, and fix $\eta>0$ small enough.
	Assume that $\varphi$ satisfies \eqref{eq:condphiconv} with $\tilde{\eta}<\eta/2$.
	For any $T>0$ and $w_0\in\mathbb{H}^1$, there exists $\varepsilon>0$ and $C=C_{w_0,T,\varphi}>0$ such that for all $n\in\mathbb{N}$ with $T/n<\varepsilon$,
	\begin{equation}
		\lvert P_{T}\varphi(w_0) - (Q^N_{(T/n)})^n\varphi|_{\mathcal{H}_N}(w_0)\rvert
		\le C\left( N^{-1} + N^{2\frac{m+1}{2}}n^{-\frac{m-1}{2}} \right).
	\end{equation}
\end{corollary}

\section{Numerical examples}
\label{sec:numerics}
We consider the problem of approximating \eqref{eq:sns-conteq} with $\nu=10^{-2}$, $w_0=0$, $d=4$, $q_j=1$, $j=1,\dots,4$, and $k_1=(1,0)$, $k_2=(-1,0)$, $k_3=(1,1)$ and $k_4=(-1,-1)$.
\cite[Example~2.5]{HairerMattingly2006} shows that the dynamics generated by this process are ergodic.
We aim to find estimates for $\mathbb{E}[\lVert w(1,0)\rVert]$, $\mathbb{E}[\lVert w(1,0)\rVert_{-1}]$ and $\mathbb{E}[\lVert w(1,0)\rVert_{+1}]$.
We remark that the first and second values equal, up to a constant, the mean enstrophy and energy, respectively.
Furthermore, control of the $\mathbb{H}^1$ norm of $w(1,0)$ means control of the $\mathbb{H}^2$ norm of $\mathcal{K}w(1,0)$, which in turn implies that we can take point evaluations of $\mathcal{K}w(1,0)$ due to the Sobolev embedding theorems in two dimensions.
This is important in the evaluation of cross correlations.

Our numerical simulations are performed using a splitting scheme, the symmetrically weighted sequential splitting
\begin{equation}
	Q^{N}_{T,n}
	:=
	\frac{1}{2}\left( (P^{N,1}_{T/n}P^{N,2}_{T/n})^n + (P^{N,2}_{T/n}P^{N,1}_{T/n})^{n} \right),
\end{equation}
going back at least to \cite[equation (25)]{Strang1963} and being of second order for problems that are smooth enough.

We apply a Monte Carlo method.
For a single realisation, we have to solve, alternatingly, a time-dependent Euler equation and an Ornstein-Uhlenbeck equation.
Note that the solution of the Ornstein-Uhlenbeck equation follows a Gaussian process, and its distribution is therefore explicitly known.
To discretise the Euler equation, we apply the standard RK4 scheme.
While the Heun method, i.e., an RK2 scheme, provides the correct order such that the entire approximation is of second order, see \cite{NinomiyaNinomiya2009}, it has suboptimal stability properties, leading to strong step size restrictions, see \cite[Section~D.2.5]{CanutoHussainiQuarteroniZang2006}.
In this regard, see also \cite{HutzenthalerJentzenKloeden2011} for issues of stability of the Euler-Maruyama scheme for equations with non-globally Lipschitz coefficients.
As we apply the FFT to determine the value of $(\mathcal{K}w_N\cdot\nabla)w_N$ efficiently, we observe aliasing effects, which are reduced by the use of the 2/3 dealiasing, see \cite[Section~3.3.2]{CanutoHussainiQuarteroniZang2007}.

To find the expected values in the definition of $P^{N,2}_{T/n}$, we use quasi-Monte Carlo integration, applying the Sobol$'$ sequences of Joe and Kuo \cite{JoeKuo2008}.
Also, instead of simulating both terms in the definition of $Q^{N}_{T,n}$, we use a Bernoulli random variable to generate either of them, retaining the order of the approximation.

\begin{figure}[htpb]
	\begin{center}
		\includegraphics[height=6cm]{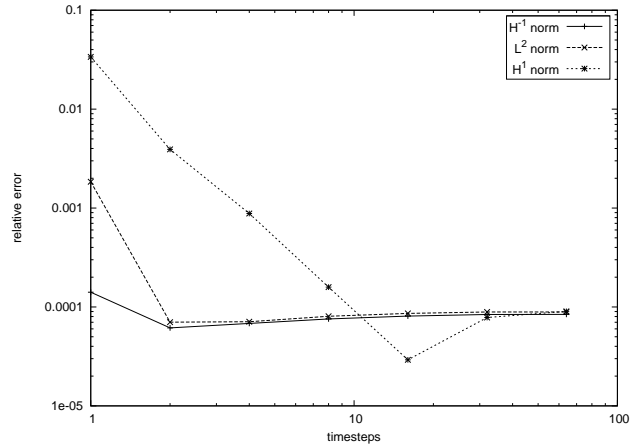}
	\end{center}
	\caption{Error plot, increasing number of timesteps}
	\label{fig:stepinc}
\end{figure}
\begin{figure}[htpb]
	\begin{center}
		\includegraphics[height=6cm]{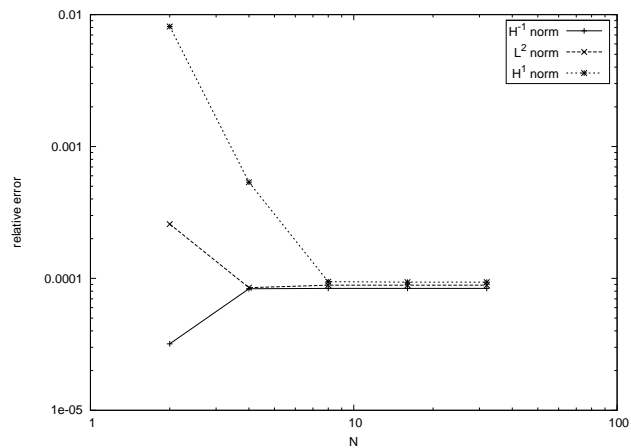}
	\end{center}
	\caption{Error plot, increasing number of Fourier modes}
	\label{fig:Ninc}
\end{figure}
\begin{figure}[htpb]
	\begin{center}
		\includegraphics[height=6cm]{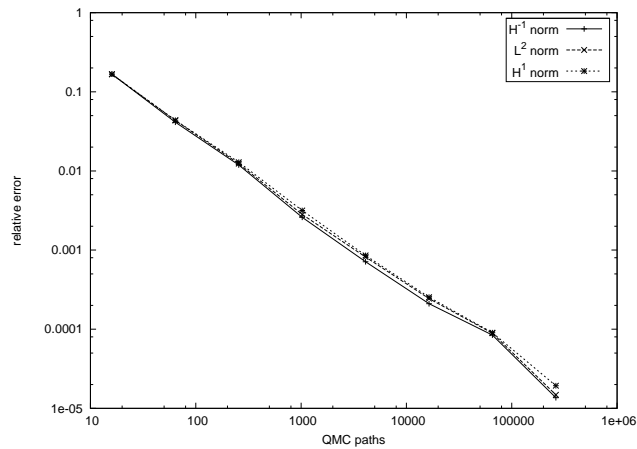}
	\end{center}
	\caption{Error plot, increasing number of quasi-Monte Carlo paths}
	\label{fig:Kinc}
\end{figure}
Figures~\ref{fig:stepinc}, \ref{fig:Ninc} and \ref{fig:Kinc} present the results of numerical calculations with increasing number of timesteps, Fourier modes, and quasi-Monte Carlo paths.
All errors are relative, and were calculated through comparison with a reference solution found using $K=2^{20}$ quasi-Monte Carlo paths, $N=32$ and $n=128$ timesteps.
There, we obtained the values $\mathbb{E}[\lVert w(1,0)\rVert_{-1}]\sim1.138449630686444$, $\mathbb{E}[\lVert w(1,0)\rVert]\sim1.319968848291092$, and $\mathbb{E}[\lVert w(1,0)\rVert_{+1}]\sim1.620419847035606$.
In Figure~\ref{fig:stepinc}, we chose the other parameters to be $K=2^{16}$ and $N=32$; in Figure~\ref{fig:Ninc}, $K=2^{16}$ and $n=128$; and in Figure~\ref{fig:Kinc}, $N=32$ and $n=64$.

We clearly see that mainly the number of quasi-Monte Carlo paths limits the attainable accuracy.
Nevertheless, with $2^{12}=4096$ paths, we obtain a relative error of less than $10^{-3}$, and that calculation took approximately 60 seconds running on 16 cores of a Primergy RX200 S6 spotting 4 Intel Xeon CPU X5650 processor, each of which provides 6 cores.
In Figure~\ref{fig:stepinc}, we observe that we obtain a rate of convergence of about $2.5$ for the $\mathbb{H}^1$ norm with respect to the number of time steps, which is even more than the theoretically predicted rate of $2$ and seems to result from the fact that we compare with numerical estimates instead of the exact value.
The solution of the model problem is smooth (see also \cite{Mattingly2002} in this regard), and indeed, Figure~\ref{fig:Ninc} exhibits spectral convergence in the number of Fourier modes.

\section{Conclusion}
\label{sec:conclusion}
We have introduced and analysed novel high order approximation schemes for the sto\-chastic Navier-Stokes equations on the 2D torus.
We prove high order accuracy in time and give precise estimates for the dependence on the order of the spectral Galerkin discretisation.
Using high order cubature paths, it is possible to attain convergence of arbitrary order in time.

From a practical point of view, the splitting schemes presented in this work have the important advantage that well-tested and robust solvers for the deterministic Navier-Stokes and Euler equations can be reused.
Furthermore, the algorithm makes increasing the dimension of the driving Brownian motion easy.
Numerical examples establish the applicability of the method to some simple, but relevant functionals.

\Appendix
\section{Proof of Proposition~\ref{prop:sns-ousupermart}}
\begin{lemma}
	For $N\sim\mathcal{N}(0,1)$, $j=1,\dots,d$, and $S$, $A$, $B\in\mathbb{R}$ with $C\in\mathbb{R}$ small enough,
	\begin{alignat}{2}
		\mathbb{E}[\exp(C( S^2 + &2SABN + (B N)^2 ))] \notag \\
		&=
		\frac{1}{(1-2CB^2)^{1/2}}\exp\left( \left(1+\frac{2CA^2B^2}{1-2CB^2}\right) CS^2 \right).
	\end{alignat}
\end{lemma}
\begin{proof}
	A direct calculation yields
	\begin{alignat}{2}
		\mathbb{E}[\exp(C( S^2 &+ 2SABN + (B N)^2 ))]
		\notag \\
		&=
		\int_{\mathbb{R}} \exp(C(S^2+2SA B y+(B y)^2)) \frac{1}{(2\pi)^{1/2}}\exp\left( -\frac{1}{2}y^2 \right)\dd y \notag \\
		&=
		\frac{1}{(2\pi)^{1/2}}\int_{\mathbb{R}} \exp\left( -\frac{1}{2}
		(1-2C B^2)\left( y-\frac{2CSAB}{1-2CB^2} \right)^2 \right)\dd y \times \notag \\
		&\phantom{=}\times
		\exp\left( \left(1+ \frac{2CA^2B^2}{1-2CB^2}\right)CS^2 \right) \notag \\
		&=
		\frac{1}{(1-2CB^2)^{1/2}}\exp\left( \left(1+\frac{2CA^2B^2}{1-2CB^2}\right)CS^2 \right), \notag
	\end{alignat}
	which proves the result.
\qquad\end{proof}
\begin{corollary}
	\label{cor:expectationexpsquarenorm}
	For independent $N_j\sim\mathcal{N}(0,1)$, $j=1,\dots,d$, and $S$, $A_j$, $B_j\in\mathbb{R}$ with $C\in\mathbb{R}$ small enough,
	\begin{alignat}{2}
		\mathbb{E}[\exp(C( S^2 + &\sum_{j=1}^{d}2SA_jB_jN_j + \sum_{j=1}^{d}(B_j N_j)^2 ))] \notag \\
		&=
		\frac{1}{\prod_{j=1}^{d}(1-2CB_j^2)^{1/2}}\exp\left( \left(1+\sum_{j=1}^{d}\frac{2CA_j^2B_j^2}{1-2CB_j^2}\right) CS^2 \right).
	\end{alignat}
\end{corollary}
{\em Proof of Proposition~\ref{prop:sns-ousupermart}}.
	Note that 
	\begin{equation}
		w^2(t,w_0)=\exp(t\varepsilon\nu\Delta)w_0 + \int_{0}^{t}\exp( (t-s)\varepsilon\nu\Delta)\sum_{j=1}^{d}q_j f_{k_j}\dd W^j_s.
	\end{equation}
	Denoting by $\lambda_j$ the eigenvalue of $f_{k_j}$ with respect to the operator $\varepsilon\nu\Delta$, $\varepsilon\nu\Delta f_{k_j}=\lambda_j f_{k_j}$, we see that
	\begin{equation}
		\int_{0}^{t}\exp( (t-s)\varepsilon\nu\Delta)Q\dd W_s
		=
		\sum_{j=1}^{d}\int_{0}^{t}\exp( (t-s)\lambda_j)q_j f_{k_j}\dd W^j_s.
	\end{equation}
	The coefficient $Z^j_t:=\int_{0}^{t}\exp( (t-s)\tilde{\lambda}_{k_j} )\dd W^j_s$ is normally distributed, more precisely, $Z^j_t\sim\mathcal{N}\left(0,\frac{1-\exp(2t\tilde{\lambda}_{k_j})}{-2\tilde{\lambda}_{k_j}}\right)$.
	In particular, with $S(t):=\exp(t\varepsilon\nu\Delta)$,
	\begin{equation}
		P^2_t\psi(w)
		=
		\mathbb{E}\biggl[ \exp\Bigl(\eta\lVert S(t)w + \sum_{j=1}^{d}q_jZ^j_tf_{k_j}\rVert_0^2\Bigr) \biggr].
	\end{equation}
	Note
	\begin{alignat}{2}
		\lVert S(t)w &+ \sum_{j=1}^{d}q_jZ^j_tf_{k_j}\rVert^2
		= 
		\lVert S(t)w\rVert^2 \notag \\
		&+ 2\sum_{j=1}^{d}\frac{\langle S(t)w,q_j f_{k_j}\rangle}{\lVert S(t)w\rVert\cdot\lVert q_j f_{k_j}\rVert} \lVert S(t)w\rVert\cdot( \lVert q_j f_{k_j}\rVert Z^j_t ) + \sum_{j=1}^{d}(\lVert q_j f_{k_j}\rVert Z^j_t)^2, 
	\end{alignat}
	and apply Corollary~\ref{cor:expectationexpsquarenorm} with $C=\eta$, $S=\lVert S(t)w\rVert$, 
	$A_j=\frac{\langle S(t)w,q_j f_{k_j}\rangle}{\lVert S(t)w\rVert_0\cdot\lVert q_j f_{k_j}\rVert_0}$ and 
	$B_j=\lVert q_j f_{k_r}\rVert\left( \frac{1-\exp(2t\tilde{\lambda}_{k_j})}{-2\tilde{\lambda}_{k_j}} \right)^{1/2}$.
	As $A_j^2\le 1$ and 
	\begin{equation}
		1-2C B_j^2
		=
		1-2\eta\lVert q_j f_{k_j}\rVert^2\frac{1-\exp(2t\tilde{\lambda}_{k_j})}{-2\tilde{\lambda}_{k_j}}
		\ge
		\exp(2\omega t)
	\end{equation}
	for $0>\omega\ge\tilde{\lambda}_{k_j}$ and $0<\eta\le\frac{-\omega}{\lVert q_j f_{k_j}\rVert^2}$ and, similarly,
	\begin{equation}
		1+\sum_{j=1}^{d}\frac{2C A_j^2B_j^2}{1-2CB_j^2}
		\le
		\exp(2\alpha t)
	\end{equation}
	for $\alpha>0$ and $\eta\le\min_{j=1,\dots,d}\frac{2\alpha}{(d-1)\lVert q_j f_{k_j}\rVert^2}$, we obtain
	\begin{equation}
		P^2_t\psi_{\eta}(w)
		\le
		\exp( -dt\omega )
		\exp( \eta \lVert w\rVert^2 )
		=\exp(-dt\omega)\psi_{\eta}(w),
	\end{equation}
	the required result. \qquad \endproof

{\bf Acknowledgements.}
The author thanks Markus Melenk and Josef Teichmann for helpful discussions.
The numerical calculations were preformed on the computing facilities of the Departement Mathematik of ETH Z\"urich.
Parts of the code were written by Dejan Velu\v{s}\v{c}ek, whom the author thanks for his support.

\bibliographystyle{siam}
\bibliography{lit}

\end{document}